\documentclass[12pt]{amsart}
\usepackage{geometry, amssymb}

\theoremstyle{plain}
\newtheorem{thm}{Theorem}[section]

\newtheorem{lem}[thm]{Lemma}

\newtheorem*{MainResult}{Main Result}
\newtheorem*{KeyLemma}{Key Lemma}
\newtheorem{cor}[thm]{Corollary}

\theoremstyle{definition}
\newtheorem{defn}{Definition}
\theoremstyle{remark}

\title{At Least Half Of All Graphs Satisfy $\chi \leq  \textstyle\frac{1}{4}\omega + \textstyle \frac{3}{4}\Delta + 1$}
\author{landon rabern}

\begin{document}
\begin{abstract}
We prove that for any graph $G$ at least one of $G$ or $\bar{G}$ satisfies
$\chi \leq \textstyle\frac{1}{4}\omega +  \textstyle\frac{3}{4}\Delta + 1$.
In particular, self-complementary graphs satisfy this bound.
\end{abstract}

\maketitle

\section{Introduction}
In \cite{reed97} Reed conjectured that

\begin{equation}\label{ReedsConjecture}
\chi \leq \left \lceil \frac{\omega + \Delta + 1}{2} \right \rceil.
\end{equation}

In the same paper he proved that there exists $\epsilon > 0$ such that
\[\chi \leq \epsilon\omega + (1 - \epsilon)\Delta + 1,\]
holds for every graph.  The $\epsilon$ used in the proof is quite small (less than $10^{-8}$).\newline

We prove the following.

\begin{MainResult}
Let $G$ be a graph.  Then at least one of $G$ or $\bar{G}$ satisfies
\[\chi \leq \textstyle\frac{1}{4}\omega +  \textstyle\frac{3}{4}\Delta + 1.\]
\end{MainResult}

To prove this we combine a result from \cite{ColoringGraphsContainingADoublyCriticalEdge} on graphs containing a doubly critical edge with the following lemma.

\begin{KeyLemma}
Every graph satisfies $\chi \leq \frac{\iota + \omega + \Delta + n + 2}{4}$.
\end{KeyLemma}

Here $\iota$ is the maximum number of singleton color classes appearing in an optimal coloring of the graph (formally defined below).

\section{Stinginess}

In  \cite{ColoringGraphsContainingADoublyCriticalEdge} it was shown that a doubly critical edge is enough to imply an upper bound on the chromatic number that is slightly weaker than Reed's conjectured upper bound.

\begin{lem}\label{BarelyStingyImpliesAlmostReed}
If $G$ is a graph containing a doubly critical edge, then
\[\chi(G) \leq \textstyle\frac{1}{3}\omega(G) + \textstyle\frac{2}{3}(\Delta(G) + 1).\]
\end{lem}

The following two lemmas were proved in \cite{reedNote} using matching theory results.
\begin{lem}\label{ReedForChiBIggerThanHalf}
If $G$ is a graph with $\chi(G) > \left \lceil \frac{|G|}{2} \right \rceil$, then 
\[\chi(G) \leq \frac{\omega(G) + \Delta(G) + 1}{2}.\]
\end{lem}

\begin{lem}\label{ReedForAlphaAtMostTwo}
If $G$ is a graph with $\alpha(G) \leq 2$, then
\[\chi(G) \leq \left\lceil\frac{\omega(G) + \Delta(G) + 1}{2}\right\rceil.\]
\end{lem}

\begin{lem}\label{ReedForAtMostTwoFrames}
Let $G$ be a graph for which every optimal coloring has all color classes of order at most $2$. Then 
\[\chi(G) \leq \left\lceil\frac{\omega(G) + \Delta(G) + 1}{2}\right\rceil.\]
\end{lem}
\begin{proof}
If $\alpha(G) \leq 2$, the result follows by Lemma \ref{ReedForAlphaAtMostTwo}.  Hence  we may assume that we have an independent set $I \subseteq G$ with $|I| \geq 3$.  Put $H = G \smallsetminus I$.  Since $G$ has no optimal coloring containing a color class of order $\geq 3$, we have $\chi(H) = \chi(G)$.  Then
\[\chi(H) = \chi(G) \geq \frac{|G|}{2} = \frac{|H| + 3}{2} > \left\lceil\frac{|H|}{2}\right\rceil.\]
Hence, by Lemma \ref{ReedForChiBIggerThanHalf}, we have 
\[\chi(G) = \chi(H) \leq \frac{\omega(H) + \Delta(H) + 1}{2} \leq  \frac{\omega(G) + \Delta(G) + 1}{2}.\]
The lemma follows.
\end{proof}

\begin{defn}
The \emph{stinginess} of a graph $G$ (denoted $\iota(G)$) is the maximum number of singleton color classes appearing in an optimal coloring of $G$.  An optimal coloring of $G$ is called \emph{stingy} just in case it has the maximum number of singleton color classes.
\end{defn}

\begin{lem}\label{StinginessPreservedInPatchedColorings}
Let $G$ be a graph and $H$ an induced subgraph of $G$.  If $\chi(G) = \chi(G \smallsetminus H) + \chi(H)$, then $\iota(G) \geq  \iota(G \smallsetminus H) + \iota(H)$.
\end{lem}
\begin{proof}
Assume that $\chi(G) = \chi(G \smallsetminus H) + \chi(H)$.  Then patching together any optimal coloring of $G \smallsetminus H$ with any optimal coloring of $H$ yields an optimal coloring of $G$.  The lemma follows.
\end{proof}

\begin{lem}\label{ChiAtMostAverageOfStinginessAndOrder}
Let $G$ be a graph. Then $\chi(G) \leq \frac{\iota(G) + |G|}{2}$.
\end{lem}
\begin{proof}
Let $C = \{I_1,\ldots,I_m, \{s_1\}, \ldots, \{s_{\iota(G)}\}\}$ be a stingy coloring of $G$.  Since $|I_j| \geq 2$ for $1 \leq j \leq m$, we have $\chi(G) \leq \iota(G) + \frac{|G| - \iota(G)}{2} = \frac{|G| + \iota(G)}{2}$.
\end{proof}

\section{Respectfully Greedy Partial Colorings}

\begin{defn}
Let $G$ be a graph.  A partial coloring $C$ of $G$ is called \emph{$r$-greedy} just in case every color class has order at least $r$.
\end{defn}

\begin{defn}
Let $G$ be a graph.  A partial coloring of $C$ of $G$ is called \emph{respectful} just in case $\chi(G \smallsetminus \cup C) = \chi(G) - |C|$.
\end{defn}

\begin{lem}\label{RespectfulAndGreedy}
Let $G$ be a graph and $C$ a respectful $3$-greedy partial coloring of $G$ with $|G \smallsetminus \cup C|$ minimal.  Then
\[\chi(G) \leq \frac{\omega(G) + \Delta(G) + 1}{2} + \frac{|C| + 1}{2}.\]
\end{lem}
\begin{proof}
Put $H = G \smallsetminus \cup C$.  By the minimality of $|H|$, every optimal coloring of $H$ has all color classes of order at most $2$.  Thus, by Lemma \ref{ReedForAtMostTwoFrames}, we have 
\[\chi(H) \leq \frac{\omega(H) + \Delta(H) + 1}{2} + \frac{1}{2}.\]

Using the minimality of $|H|$ again, we see that every vertex of $H$ is adjacent to at least one vertex in each element of $C$.  Hence $\Delta(H) \leq \Delta(G) - |C|$.  Putting it all together, we have
\begin{align*}
\chi(G) &= \chi(H) + |C| \\
&\leq \frac{\omega(H) + \Delta(H) + 1}{2} + \frac{1}{2} + |C| \\
&\leq \frac{\omega(H) +  \Delta(G) - |C| + 1}{2} + \frac{1}{2} + |C| \\
&\leq \frac{\omega(G) +  \Delta(G) - |C| + 1}{2} + \frac{1}{2} + |C| \\
&= \frac{\omega(G) +  \Delta(G) + 1}{2} + \frac{|C| + 1}{2}. \\
\end{align*}
\end{proof}

\begin{KeyLemma}
Every graph satisfies $\chi \leq \frac{\iota + \omega + \Delta + n + 2}{4}$.
\end{KeyLemma}
\begin{proof}
Let $C$ be a respectful $3$-greedy partial coloring of a graph $G$ with $|G \smallsetminus \cup C|$ minimal. Since $\chi(G \smallsetminus \cup C) = \chi(G) - |C|$ we have $\iota(G \smallsetminus \cup C) \leq \iota(G)$ (by Lemma \ref{StinginessPreservedInPatchedColorings}).  Applying Lemma \ref{ChiAtMostAverageOfStinginessAndOrder} yields

\begin{align*}
\chi(G) &= \chi(G \smallsetminus \cup C) + |C| \\
&\leq  \frac{\iota(G) + |G| - |\cup C|}{2} + |C| \\
&\leq  \frac{\iota(G) + |G| -|C|}{2}.
\end{align*}

Adding this inequality with the inequality in Lemma \ref{RespectfulAndGreedy} gives

\[2\chi(G) \leq \frac{\iota(G) + \omega(G) + \Delta(G) + |G| + 2}{2}.\]
The lemma follows.
\end{proof}

\section{The Main Results}

\begin{thm}\label{MainDisjunct}
Let $G$ be a graph. Then at least one of the following holds,
\begin{enumerate}
\item $\chi(G) \leq \textstyle\frac{1}{3}\omega(G) + \textstyle\frac{2}{3}(\Delta(G) + 1)$,
\item $\chi(G) \leq \frac{\omega(G) + |G| + \Delta(G) + 3}{4}$.
\end{enumerate}
\end{thm}
\begin{proof}
Assume that (1) does not hold.  Then, by Lemma \ref{BarelyStingyImpliesAlmostReed}, we have $\iota(G) < 2$.  Applying the Key Lemma gives
\[\chi(G) \leq \frac{1 + \omega(G) + \Delta(G) + |G| + 2}{4}.\]  
The theorem follows.
\end{proof}

\begin{cor}
Let $G$ be a graph satisfying $\Delta \geq \frac{n}{2}$.  Then $G$ also satisfies
\[\chi \leq \textstyle\frac{1}{4}\omega +  \textstyle\frac{3}{4}(\Delta + 1).\]
\end{cor}
\begin{proof}
By Theorem \ref{MainDisjunct}, $G$ satisfies
\begin{align*}
\chi &\leq \max\{\textstyle\frac{1}{3}\omega + \textstyle\frac{2}{3}(\Delta + 1), \frac{\omega + n + \Delta + 3}{4}\}  \\
&\leq \max\{\textstyle\frac{1}{4}\omega + \textstyle\frac{3}{4}(\Delta + 1), \frac{\omega + n + \Delta + 3}{4}\}  \\
&\leq \max\{\textstyle\frac{1}{4}\omega + \textstyle\frac{3}{4}(\Delta + 1), \frac{\omega + 3\Delta + 3}{4}\}  \\
&=\textstyle\frac{1}{4}\omega + \textstyle\frac{3}{4}(\Delta + 1).
\end{align*}
\end{proof}

We would like to find an upper bound on the chromatic number that must hold for a graph or its complement.  The previous corollary is not quite good enough for this purpose since it doesn't handle $\frac{n - 1}{2}$-regular graphs.  Instead, we use the following.

\begin{cor}\label{LargeDeltaImpliesQuarterBound}
Let $G$ be a graph satisfying $\Delta \geq \frac{n - 1}{2}$.  Then $G$ also satisfies
\[\chi \leq \textstyle\frac{1}{4}\omega +  \textstyle\frac{3}{4}\Delta + 1.\]
\end{cor}
\begin{proof}
By Theorem \ref{MainDisjunct}, $G$ satisfies
\begin{align*}
\chi &\leq \max\{\textstyle\frac{1}{3}\omega + \textstyle\frac{2}{3}(\Delta + 1), \frac{\omega + n + \Delta + 3}{4}\}  \\
&\leq \max\{\textstyle\frac{1}{4}\omega + \textstyle\frac{3}{4}(\Delta + 1), \frac{\omega + n + \Delta + 3}{4}\}  \\
&\leq \max\{\textstyle\frac{1}{4}\omega + \textstyle\frac{3}{4}(\Delta + 1), \frac{\omega + 3\Delta + 4}{4}\}  \\
&=\textstyle\frac{1}{4}\omega + \textstyle\frac{3}{4}\Delta + 1.
\end{align*}
\end{proof}

Since every graph satisfies $\Delta + \bar{\Delta} \geq \Delta + n - 1 - \Delta = n - 1$, combining the pigeonhole principle with Corollary \ref{LargeDeltaImpliesQuarterBound} proves the following.

\begin{MainResult}
Let $G$ be a graph.  Then at least one of $G$ or $\bar{G}$ satisfies
\[\chi \leq \textstyle\frac{1}{4}\omega +  \textstyle\frac{3}{4}\Delta + 1.\]
\end{MainResult}

\section{Some Related Results}
In \cite{ColoringAndTheLonelyGraph} the following was proven.

\begin{lem}\label{VeryStingyImpliesReed}
If $G$ is a graph with $\iota(G) > \frac{\omega(G)}{2}$, then
\[\chi(G) \leq \frac{\omega(G) + \Delta(G) + 1}{2}.\]
\end{lem}

\begin{thm}
Let $G$ be a graph. Then at least one of the following holds,
\begin{enumerate}
\item $\chi(G) \leq \frac{\omega(G) + \Delta(G) + 1}{2}$,
\item $\chi(G) \leq \frac{3}{8}\omega(G) + \frac{|G| + \Delta(G) + 2}{4}$.
\end{enumerate}
\end{thm}
\begin{proof}
Assume that (1) does not hold.  Then, by Lemma \ref{VeryStingyImpliesReed}, we have $\iota(G) \leq \frac{\omega(G)}{2}$.  Applying the Key Lemma gives
\[\chi(G) \leq \frac{\frac{\omega(G)}{2} + \omega(G) + \Delta(G) + |G| + 2}{4}.\]  
The theorem follows.
\end{proof}

\end{document}